\documentclass[10pt]{article}
\title{A proof of Lov\'asz's theorem on maximal lattice-free sets}
\author{Gennadiy Averkov\footnote{University of Magdeburg, 39106 Magdeburg, Germany, email: averkov@math.uni-magdeburg.de}}
\usepackage[active]{srcltx}
\usepackage{amsmath,amssymb,amsthm,enumerate}
\usepackage{layout}

\newtheorem{nn}{}
\newtheorem{lemma}[nn]{Lemma}
\newtheorem{theorem}[nn]{Theorem}

\newcommand{\thmheader}[1]{{\upshape({#1}).}}
\newcommand{\real}{\mathbb{R}}
\newcommand{\eps}{\varepsilon}
\newcommand{\integer}{\mathbb{Z}}
\newcommand{\sprod}[2]{\left<{#1}\,,\,{#2}\right>}
\newcommand{\natur}{\mathbb{N}}

\newcommand{\rmcmd}[1]{\mathop{\mathrm{#1}}}
\newcommand{\intr}{\rmcmd{int}}
\newcommand{\rec}{\rmcmd{rec}}

\newcommand{\setcond}[2]{\left\{{#1}\,:\,{#2} \right\}}
\newcommand{\cl}{\rmcmd{cl}}

\newcommand{\lin}{\rmcmd{lin}}
\newcommand{\dotvar}{\,\cdot\,}

\begin{document}

\maketitle

\begin{abstract}
	Let $K$ be a maximal lattice-free set in $\real^d$, that is, $K$ is convex and closed subset of $\real^d$, the interior of $K$ does not cointain points of $\integer^d$ and $K$ is inclusion-maximal with respect to the above properties. A result of Lov\'asz assert that if $K$ is $d$-dimensional, then $K$ is a polyhedron with at most $2^d$ facets, and the recession cone of $K$ is spanned by vectors from $\integer^d$. A first complete proof of mentioned Lov\'asz's result has been published in a paper of Basu, Conforti, Cornu\'ejols and Zambelli (where the authors use Dirichlet's approximation as a tool). The aim of this note is to give another proof of this result. Our proof relies on Minkowki's first fundamental theorem from the gemetry of numbers. We remark that the result of Lov\'asz is relevant in integer and mixed-integer optimization.
\end{abstract}

\newtheoremstyle{itsemicolon}{}{}{\mdseries\rmfamily}{}{\itshape}{:}{ }{}

\newtheoremstyle{itdot}{}{}{\mdseries\rmfamily}{}{\itshape}{:}{ }{}

\theoremstyle{itdot}

\newtheorem*{msc*}{2010 Mathematics Subject Classification}

\begin{msc*}
	Primary 90C11;  Secondary 90C10, 52A01,  52C07
\end{msc*}

\newtheorem*{keywords*}{Keywords}

\begin{keywords*}
	cutting plane; lattice-free set; Lov\'asz's theorem
\end{keywords*}

\section{Introduction}

A set $K \subseteq \real^d$ is called \emph{lattice-free} if $K$ is closed, convex and the interior of $K$ does not contain points of $\integer^d$. A lattice-free set $K$ in $\real^d$ is called \emph{maximal} if $K$ is not properly contained in another lattice-free set. In  \cite{MR1114315} Lov\'asz formulated a result which provides a description of maximal-lattice free sets. In order to state the result of Lov\'asz we need the notion of recession cone. If $K$ is a nonempty closed convex set in $\real^d$, then the recession cone $\rec(K)$ of $K$ is the set of all vectors $u \in \real^d$ such that the translation of $K$ by vector $u$ is a subset of $K$.

\begin{theorem} \thmheader{\cite[\S3]{MR1114315}}
	\label{lov:thm}
	Let $K$ be a $d$-dimensional maximal lattice-free set in $\real^d$ and let $r$ be the dimension of $\rec(K)$. Then the following conditions hold:
	\begin{enumerate}[I.]
		\item $K$ is a polyhedron with at most $2^{d-r}$ facets;
		\item $\rec(K)$ is a linear space spanned by $r$ vectors from $\integer^d$;
		\item the relative interior of every facet of the polyhedron $K$ contains at least one point from $\integer^d$.
	\end{enumerate}
\end{theorem}

Our formulation above is slightly different than that given in  \cite[\S3]{MR1114315}. One can also formulate a converse implication: it is clear that every $d$-dimensional lattice-free polyhedron $K$ in $\real^d$ which satisfies condition III of Theorem~\ref{lov:thm} is maximal.

In \cite[\S3]{MR1114315} Lov\'asz only gave a brief sketch of a possible proof of Theorem~\ref{lov:thm}. A complete proof of Theorem~\ref{lov:thm} has recently been presented in \cite[Theorem~2.2]{MR2724071}. The authors of \cite{MR2724071} also proved the following counterpart of Theorem~\ref{lov:thm} dealing with maximal lattice-free sets of dimension less than $d$.

\begin{theorem} \thmheader{\cite[Theorem~2.2.(ii)]{MR2724071}}
	\label{lov:deg:thm}
	A subset $K$ of $\real^d$ is a maximal lattice-free set $K$ of dimension less than $d$ if and only if $K$ is a translate of a $(d-1)$-dimensional linear space $L$ such that $L \ne \lin(L \cap \integer^d)$.
\end{theorem}

The aim of this note is to provide a short argument by which one can derive Theorems~\ref{lov:thm} and \ref{lov:deg:thm} from Minkowski's first fundamental theorem (the corresponding proofs from \cite{MR2724071} use Dirichlet's approximation). We remark that integer maximal lattice-free sets are relevant in integer and mixed-integer optimization. In fact, such sets play an important role in the cutting-plane theory (for more details see \cite{MR2480507,MR0290793,MR2724071}).

\section{Proofs}

We shall use basic facts and notions from convex geometry (see \cite{MR1216521}) and basic information on lattices (see \cite{MR893813}). The zero vector of $\real^d$ is denoted by $o$. The standard scalar product and standard Euclidean norm of $\real^d$ are denoted by $\sprod{\dotvar}{\dotvar}$ and $\|\dotvar\|$, respectively.  By $\intr$ (resp. $\cl$) we denote the interior (resp. closure) operation in the Euclidean topology. The notation $\lin$ stands for the linear hull. If $T \subseteq \real$ and $u \in \real^d$, let $T u := \setcond{t u}{t \in T}$. For $a \in \real^d, X, Y \subseteq \real^d$ we use the standard notations $X +Y := \setcond{x+y}{x \in X, \ y \in Y}$, $X - Y := \setcond{x-y}{x \in X, \ y \in Y}$, $-X := \setcond{-x}{x \in X}$ and $a+X := \setcond{a+x}{x \in X}$.

\begin{theorem} \thmheader{Minkowski's first fundamental theorem, \cite[\S5]{MR893813}} \label{mink:thm}
	Let $K$ be a $d$-dimensional bounded and closed convex set in $\real^d$ which is symmetric in the origin (that is, $K=-K$). Let the volume of $K$ be at least $2^d$. Then there exists a vector $z \in \integer^d \setminus \{o\}$ belonging to $K$.
\end{theorem}

Let $K$ be as in Theorem~\ref{mink:thm} and let $t \in \natur$ be such that the volume of $\frac{1}{t} K$ is at least $2^d$ (which means that the volume of $K$ is at least $(2 t)^d$). From Theorem~\ref{mink:thm} we deduce that there exists $z \in (t\integer^d) \setminus \{o\}$ belonging to $K$. The latter conclusion is equivalent to the version of Minkowski's theorem for the lattice $t \integer^d$ in place of $\integer^d$.

\begin{lemma} \label{lin:rec:cone:lem}
	Let $K \subseteq \real^d$ be a $d$-dimensional lattice-free set. Then $K - \rec(K) = K + \lin(\rec(K))$ is lattice-free.
\end{lemma}
\begin{proof}
	It is not difficult to verify the equality $\rec(K)-\rec(K) = \lin(\rec(K))$. The latter equality yields $K - \rec(K) = (K + \rec(K)) - \rec(K) = K + \lin(\rec(K))$.
	We show that $K-\rec(K)$ is lattice-free by contradiction. Assume $K-\rec(K)$ is  not lattice-free. Then $\intr(K - \rec(K)) \subseteq \intr(K) - \rec(K)$ contains a point $z$ of $\integer^d$. Hence there exists $u \in \rec(K)$ such that $z+ u \in \intr(K)$. Let $B$ be a sufficiently small closed Euclidean ball with center at $o$ such that $z+u+B \subseteq \intr(K)$. We introduce a parameter $t \in \natur$, which will be fixed later. Let us choose $N = N(t) >0$ large enough to ensure that the volume of  $[-N,N] u + B$ is at least $(2t)^d$. By Minkowski's theorem, there exists $w \in (t \integer^d) \setminus \{o\}$ with $w \in [-N,N]u+B$. Possibly replacing $w$ by $-w$, we get $w \in [0,N]u+B$. Since $w \in (t \integer^d) \setminus \{o\}$, we obtain $\|w\|\ge t$. Thus, choosing $t$ large enough we obtain $z+w \in z + [1,+\infty) u +B$, where  $ z + [1,+\infty) u +B \subseteq \intr(K)$. It follows $z+w \in \integer^d \cap \intr(K)$ contradicting the assumption on $K$.
\end{proof}

\begin{lemma} \label{rat:lin:rec:cone:lem}
	Let $K \subseteq \real^d$ be lattice-free and $d$-dimensional and let $L:=\rec(K)$ be a linear space. Let $M \subseteq \cl(\integer^d + L)$ be a linear space. Then $K+M$ is lattice-free.
\end{lemma}
\begin{proof}
	The assertion is verified by the following chain of implications:
	\begin{align*}
		\intr(K) \cap \integer^d = \emptyset & & & \Rightarrow & (\intr(K) +L) \cap \integer^d =&\emptyset 	& & \Rightarrow & \intr(K) \cap (\integer^d + L) = &\emptyset \\
		& & & \Rightarrow & \intr(K) \cap \cl(\integer^d + L)  = &\emptyset & & \Rightarrow & \intr(K) \cap (\integer^d+ M) = &\emptyset \\
		& & & \Rightarrow &  (\intr(K) + M) \cap \integer^d = &\emptyset & & \Rightarrow &   (\intr(K+ M )) \cap \integer^d = &\emptyset.
	\end{align*}
\end{proof}

\begin{lemma} \label{rat:space:lem}
	Let $L \subseteq \real^d$ be a linear space such that $L \ne \lin(\integer^d \cap L)$. Then there exists a line $l$ through the origin  satisfying $l \subseteq \cl(\integer^d + L)$ and  $l \not\subseteq L$.
\end{lemma}
\begin{proof}
	Choose $u \in L \setminus \{o\}$  orthogonal to $\lin(\integer^d \cap L)$. Let $B$ be a closed Euclidean ball in $\real^d$ with center in $o$ and of radius $<1$. By construction one has 
	\begin{equation} \label{inf:cyl:eq}
		(B + \real u) \cap L \cap \integer^d = \{o\}
	\end{equation}
Consider an arbitrary $t \in \natur$. There exists a sufficiently large $N = N(t) >0$ such that the volume of $[-N,N] u + \frac{1}{t} B$ is larger than $2^d$. By Minkowski's theorem, there exists $z_t \in \integer^d \setminus \{o\}$ with $z_t \in  [-N,N] u + \frac{1}{t} B$. In view of \eqref{inf:cyl:eq}, $z_t \not\in L$. Let $x_t$ be the orthogonal projection of $z_t$ onto $L$. By construction 
\begin{equation} \label{dist:to:proj:eq}
	0<\|z_t - x_t\|< \frac{1}{t}.
\end{equation}
For an appropriate infinite subset $T$ of $\natur$, the unit vector $(z_t - x_t)/\|z_t - x_t\|$ converges to some unit vector $a$, as $t$ goes to infinity over points of $T$. Taking into account \eqref{dist:to:proj:eq} we see that every vector $ \lambda a$ with $\lambda \in \real$ can be approximated by a vector of the form $n_t (z_t - x_t)$, with $t \in T$ and an appropriate $n_t \in \integer$, arbitrarily well. Hence $a \in \cl(\integer^d+L)$. The assertion follows by choosing $l := \real a$.
\end{proof}

The following lemma is well-known (see also \cite{Bell77,MR0387090,MR0452678,MR1114315}).

\newcommand{\modulo}{\mathrm{mod}}

\begin{lemma} \thmheader{Parity lemma} \label{parity:lem}
	Let $w_1,\ldots,w_m \in \integer^d$ with $m \in \natur$ and $m > 2^d$. Then there exist $1 \le i < j \le m$ such that $(w_i+w_j) / 2 \in \integer^d$.
\end{lemma}
\begin{proof}
	Since $\integer^d / (2 \integer^d) = (\integer / 2 \integer)^d$ has cardinality $2^d$, there exist $1 \le i < j \le m$ such that $w_i \equiv w_j \, (\modulo \, 2 \integer^d)$. Hence $(w_i+w_j)/2 \in \integer^d$.
\end{proof}

\begin{proof}[Proof of Theorem~\ref{lov:thm}]
	First we consider the case that $K$ is bounded (that is, $r=0$). In this case arguments from \cite{MR2724071} can be employed. In order to give a self-contained presentation we repeat these arguments. In the case of bounded $K$ we only need to verify II and III. Choose $N > 0 $ sufficiently large to ensure $K \subseteq [-N,N]^d$. Let $[-N,N]^d \cap \integer^d = \{z_1,\ldots,z_n\}$, where $n  \ge 0$. By separation theorems, for every index $i$ with $1 \le i \le n$ there exists a closed halfspace $H_i$ such that $K \subseteq H_i$ and $z_i$ is not in the interior of $H_i$. Hence $K \subseteq H_1 \cap \cdots \cap H_n \cap [-N,N]^d$ and we even have $K = H_1 \cap \cdots \cap H_n \cap [-N,N]^d$, since otherwise $K$ were not maximal. This shows that $K$ is a polytope. Let $m$ be the number of facets of $K$ and let $K = \setcond{x \in \real^d}{\sprod{a_1}{x} \le t_1,\ldots,\sprod{a_m}{x} \le t_m}$ for some $a_1,\ldots,a_m \in \real^d \setminus \{o\}$ and $t_1,\ldots,t_m \in \real^d$. Assume III were not valid, that is, the relative interior of some facet, say the facet $K \cap \setcond{x \in \real^d}{\sprod{a_1}{x} = t_1}$, contains no point of $\integer^d$. For every $\eps>0$, the polyhedron 
	\[K_\eps:=\setcond{x \in \real^d}{\sprod{a_1}{x} \le t_1 + \eps,\sprod{a_2}{x} \le t_2,\ldots,\sprod{a_2}{x} \le t_m}
	\]
is bounded. Thus, $\intr(K_\eps) \cap \integer^d$ is finite. By the assumption, every  $z \in \intr(K_\eps) \cap \integer^d$ satisfies $\sprod{a_1}{x} > t_1$. Hence, if $\eps>0$ is small enough, $K_\eps$ is lattice-free, which is a contradiction to the maximality of $K$. This verifies III. Let us show the bound on the number of facets formulated in I. We introduce $w_1,\ldots,w_m \in \integer^d$ by choosing an integer vector in the relative interior of each of the $m$ facets of $K$. If $m>2^d$, then by the parity lemma (Lemma~\ref{parity:lem}) there exist $1 \le i < j \le m$ such that $(w_i+w_j)/2 \in \integer^d$. The integer point $(w_i+w_j)/2$ lies in the interior of $K$, a contradiction. This shows $m \le 2^d$.

	Now let us consider the case of unbounded $K$. By Lemma~\ref{lin:rec:cone:lem}, $\rec(K)$ is a linear space. By Lemmas~\ref{rat:lin:rec:cone:lem} and \ref{rat:space:lem}, the linear space $\rec(K)$ is spanned by $r$ vectors from $\integer^d$. Thus, II holds. Let $L:=\rec(K)$. Let us fix a basis $u_1,\ldots,u_r$ of the lattice $L \cap \integer^d$ and extend this basis to a basis $u_1,\ldots,u_d$ of the lattice $\integer^d$. The linear transformation
	\[
		A: \ t_1 u_1 + \cdots + t_d u_d \in \real^d \mapsto (t_1,\ldots,t_d) \in \real^d
	\]
	is bijective and maps $\integer^d$ onto $\integer^d$ (that is, $A$ is unimodular). By construction, $A$ maps $L$ onto $\real^r \times \{o\}$. Consequently, $A$
	maps $K$ onto a maximal lattice-free polyhedron of the form $\real^r \times K'$. Straightforward vericiation shows that $K' \subseteq \real^{d-r}$ is maximal lattice-free. Thus, conditions I and III for the unbounded $K$ follow by application of I and III to the $(d-r)$-dimensional bounded maximal lattice-free set $K'$.
\end{proof}

\begin{proof}[Proof of Theorem~\ref{lov:deg:thm}]
	Let $K \subseteq \real^d$ be maximal lattice-free and of dimension $<d$. It can be seen immediately that $K$ is a hyperplane, i.e., $K = a + L$, for some $a \in \real^d$ and a linear space $L$ of dimension $d-1$. Let us show that $L \ne \lin(L \cap \integer^d)$. Assume the contrary, that is, $L= \lin(L \cap \integer^d)$. Changing a basis of the lattice $\integer^d$, (as in Proof of Theorem~\ref{lov:thm}) we can assume $L=\real^{d-1} \times \{0\}$. Then $K =  \real^{d-1} \times \{\alpha \} $, where $\alpha$ is the last component of $a$. It follows that the lattice-free set $K$ is not maximal (since the lattice-free subset $\{\alpha\}$ of $\real^1$ is not maximal). This is a contradiction. 

	Conversely, let $K = a + L$, where $a \in \real^d$ and $L \ne \lin(L \cap \integer^d)$. Let us show that the lattice-free set $K$ is maximal. If $K$ were not maximal, one could find a lattice-free set $K'$ properly containing $K$. The set $K'$ is $d$-dimensional. Choose a line $l$ as in Lemma~\ref{rat:space:lem}. By Lemma~\ref{rat:lin:rec:cone:lem}, $K' + l$ is lattice-free. But $K' + l \supseteq K + l = a+ L + l = a + \real^d = \real^d$, which is a contradiction.
\end{proof}


\providecommand{\bysame}{\leavevmode\hbox to3em{\hrulefill}\thinspace}
\providecommand{\MR}{\relax\ifhmode\unskip\space\fi MR }
\providecommand{\MRhref}[2]{%
  \href{http://www.ams.org/mathscinet-getitem?mr=#1}{#2}
}
\providecommand{\href}[2]{#2}

\end{document}